\documentclass[12pt]{amsart}

 \usepackage{amsmath,amsthm,amsfonts,amssymb,verbatim}
 \usepackage{url}
 \usepackage{graphicx}
 \usepackage[all]{xy}
 \usepackage{wrapfig}
 \usepackage{picins}
 \usepackage{pinlabel}
 \usepackage{subfigure}
 
\newcommand\ons{Ozsv{\'a}th and Szab{\'o}}
\newcommand\os{{Ozsv{\'a}th-Szab{\'o}}}

\newcommand{\bZ}{\mathbb{Z}}
\newcommand{\bQ}{\mathbb{Q}}

\newcommand{\pq}{\frac{p}{q}}

\newcommand{\into}{\hookrightarrow}

\newcommand{\HFhat}{\widehat{\operatorname{HF}}}

\newcommand{\Khred}{\widetilde{\operatorname{Kh}}}
\newcommand{\rk}{\operatorname{rk}}

\newtheorem{theorem}{Theorem}%[section]

\newtheorem{definition}[theorem]{Definition}

\newtheorem{corollary}[theorem]{Corollary}
\newtheorem{proposition}[theorem]{Proposition}
\newtheorem{remark}[theorem]{Remark}
\newtheorem{lemma}[theorem]{Lemma}

\newtheorem*{namedtheorem}{\theoremname}
\newcommand{\theoremname}{testing}

\title[Does Khovanov homology detect the unknot?]{Does Khovanov homology detect the unknot?}
\date{May 28, 2008}
\author[Matthew Hedden]{Matthew Hedden}
\address{Department of Mathematics, Massachusetts Institute of Technology, MA.}
\email{mhedden@math.mit.edu}
\urladdr{http://www-math.mit.edu/~mhedden}

\author[Liam Watson]{Liam Watson}
\thanks{Matthew Hedden was partially supported by NSF Grant DMS-0706979}
\thanks{Liam Watson was supported by a Canadian Graduate Scholarship (NSERC)}

\address{D\'epartement de Math\'ematiques, Universit\'e du Qu\'ebec \`a Montr\'eal,  Montr\'eal Canada.}
\email{liam.watson@cirget.ca}
\urladdr{http://www.cirget.uqam.ca/~liam}

\begin{document}

\begin{abstract}

We determine a wide class of knots, which includes unknotting number one knots, within which Khovanov homology detects the unknot.
A corollary is that the Khovanov homology of many satellite knots, including the Whitehead double, detects the unknot. 
\end{abstract}

\maketitle

\section{Introduction}

The question of the existence of a non-trivial knot with trivial Jones polynomial \cite{Jones1985} has received considerable attention since the discovery of this revolutionary knot invariant. While the question remains open, Khovanov's categorification of the Jones polynomial \cite{Khovanov2000} gives rise to a natural reformulation: Is there a non-trivial knot for which the reduced Khovanov homology has rank 1? We establish a class of knots for which the answer is no.

\begin{theorem}\label{thm:main}  Suppose $K\into S^3$ has tangle unknotting number one. Then $\rk\Khred(K)=1$ if  and only if $K$ is the unknot.
\end{theorem}

Here, $\Khred$ denotes the reduced Khovanov homology introduced in \cite{Khovanov2003}; we work with $\bZ_2$ coefficients throughout. A knot has tangle unknotting number one if the unknot may be obtained from it by exchanging one rational tangle for another (see Definition \ref{def:tangleunknotting}).  This appears to be a rather large class of knots. In particular, unknotting number one knots have tangle unknotting number one.

The above theorem becomes particularly interesting in light of the following corollary, which indicates that the Khovanov homology of many satellite knots can be used to detect the unknot.  To describe it, let $P(K)$ be the satellite knot of $K$ with pattern $P$.  By pattern, we mean that $P$ is the knot in the solid torus which is identified with the neighborhood of $K$ in the satellite construction.

\newpage
\begin{corollary}\label{cor:satellite} Let $P\hookrightarrow S^1\times D^2$ be a knot in the solid torus.  Suppose that 
\begin{itemize}
\item  For any $K$, $P(K)$ has tangle unknotting number one.
\item $P(K)\simeq U$ if and only if  $K\simeq U$, where $U$ is the unknot.
\end{itemize}
Then $\rk\Khred(P(K))=1$ if and only if $K\simeq U$.  In particular, the reduced Khovanov homology of the satellite operation defined by $P$ detects the unknot.  
\end{corollary}

For instance, the above corollary shows that the Khovanov homology of the untwisted Whitehead double, the $(2,1)$-cable or, more generally, the infinite family of satellites defined by the figure in Section \ref{sec:cor}, all detect the unknot.   

Theorem \ref{thm:main} relies on the fact that there is a spectral sequence from the reduced Khovanov homology of $K$ to the \os \ Floer homology  of the branched double cover of $K$.  Tangle unknotting number one knots are those knots whose branched double covers can be obtained by surgery on (other) knots in the three-sphere (see Lemma \ref{lemma:tanglebranched}).  The theorem then follows from formulas which compute the Floer homology of manifolds obtained by Dehn surgery on a knot in terms of its knot Floer homology invariants, together with known properties of these latter invariants (e.g. they detect whether the knot is fibered). 

The corollary should be compared to \cite{Hedden2008} which uses the fact that Floer homology detects the Thurston norm to show that the Khovanov homology of the $2$-cable link detects the unknot.  In particular, it would be interesting to understand the full extent to which a combination of surgery formula techniques  and analysis of the Thurston norm of branched covers can be applied to understand  whether Khovanov homology detects the trivial link.

\subsection*{Acknowlegement} This work benefitted from useful discussions with Steve Boyer.  
\section{Tangles and two-fold branched covers}

Let $\tau\into B^3$ be a pair of embedded arcs in the 3-ball, intersecting the boundary $\partial B^3 = S^2$ transversally in 4 points. The pair $T=(B^3,\tau)$ is referred to as a tangle. We consider tangles up to homeomorphism in the sense of Lickorish \cite{Lickorish1981}; such a homeomorphism need not fix the boundary in general.

Tangles arise naturally as component pieces of knots. Given a knot $K\into S^3$ and an embedding $S^2\into S^3$ such that $S^2$  intersects $K$ transversely in 4 points, the resulting decomposition of $S^3$ into 3-balls restricts to a decomposition of $K$ into tangles, denoted $K=T_0\cup T_1$.

If $\Sigma_K^2$ denotes the two-fold branched cover of $S^3$ branched over the knot $K$, then the decomposition $K=T_0\cup T_1$ lifts to a decomposition of $\Sigma_K^2$ into manifolds with torus boundary. The two-fold branched cover of a tangle is always a manifold with torus boundary, and a tangle is {\em rational} if and only if this cover is a solid torus \cite{Lickorish1981}. 

\begin{definition}\label{def:tangleunknotting} A knot $K$ has tangle unknotting number one if there is a decomposition $K=T_0\cup T_1$ with the property that $T_0\cup T_2$ is the unknot, where $T_1$ and $T_2$ are rational tangles. \end{definition}  

Notice that if $K$ has tangle unknotting number one then $\Sigma_K^2\cong S^3_{p/q}(K')$ where $S^3_{p/q}(K')$ denotes $\pq$-framed surgery on a knot $K'\into S^3$. Moreover, $K'$ must be 
{\em strongly invertible} i.e. there is an orientation-preserving involution of $S^3$ which preserves $K'$ as a set and reverses the orientation of $K'$.  Conversely, to any strongly invertible knot $K'$ we may associate a tangle $T$ by taking the quotient of the knot complement $S^3\smallsetminus\nu({K'})$ by the $\bZ_2$-action from the strong inversion. Recall that a strong inversion on $K'$ gives an involution on $S^3\smallsetminus\nu({K'})$ fixing a pair of arcs intersecting the boundary transversally in 4 points \cite{Waldhausen1969}. The strands $\tau\into B^3$ of the tangle then correspond to the image of the fixed point set in the quotient. In summary we have the following.

\begin{lemma}\label{lemma:tanglebranched}
The two-fold branched cover of a knot $K\hookrightarrow S^3$ is obtained by surgery on a knot $K'\hookrightarrow S^3$ if and only if $K$ has tangle unknotting number one.  Moreover, $K'$ is strongly invertible.  
\end{lemma}
\section{The proof of Theorem \ref{thm:main}}

The proof relies heavily on the machinery of Heegaard-Floer homology introduced by \ons\ \cite{OSz2004}. We will make use of the ``hat" version of the theory $\HFhat$, and take coefficients in $\bZ_2$. 

In \cite[Corollary 1.2]{OSz2005-1}, \ons\ show the following:
\[\left|H_1(\Sigma_L^2;\bZ)\right|\le\rk\HFhat(\Sigma_L^2)\le\rk\Khred(L)\]
This inequality follows from the fact that that there is a spectral sequence with $E_2$ term given by the reduced Khovanov homology of the mirror of $L$ converging to $\HFhat(\Sigma_L^2)$ \cite[Theorem 1.1]{OSz2005-1}. 

Now suppose that $K$ is a tangle unknotting number one knot. Then $\Sigma_K^2\cong S^3_{p/q}(K')$, and by passing to the mirror image if necessary we may assume that $\pq>0$ (notice that since we are considering knots the case $\pq=0$ is omitted). In this setting we obtain \[p\le\rk\HFhat(S^3_{p/q}(K'))\le\rk\Khred(K)\] and Theorem \ref{thm:main} follows immediately if $p>1$. Therefore we may reduce to the case of $\frac{1}{q}$-framed surgeries so that $S^3_{1/q}(K')$ is a $\bZ$-homology sphere. Specifically, our task is to consider the case $\rk\HFhat(S^3_{1/q}(K'))=1$. That is, the case when surgery on a knot in $S^3$ yields a $\bZ$-homology sphere L-space. Recall that a $\bQ$-homology sphere $Y$ is an L-space whenever $\left|H_1(Y;\bZ)\right|=\rk\HFhat(Y)$. In the spirit of \cite{OSz2006}, we have the following.
\begin{proposition}\label{prp:main} Suppose $K'$ is a non-trivial knot and $S^3_{1/q}(K')$ is an $L$-space.  Then $q=1$ and $K'$ is the trefoil knot. \end{proposition}
\begin{proof}
If $S^3_{1/q}(K')$ is an L-space,  \cite[Proposition 9.5]{OSz2005-3} implies that $S^3_{1}(K')$ is an L-space as well. Applying \cite[Corollary 8.5]{KMOSz2007} (see also \cite{Hedden2007,Rasmussen2007}) gives the bound $g\le 1$, where $g$ denotes the genus of the knot $K'$. Since $K'$ is non-trivial by assumption, $g=1$ and it follows from \cite{Ghiggini2007} (see also \cite{Ni2007}) that $K'$ is the trefoil.   
\end{proof}

\begin{remark}\label{rmk:main}An alternate approach comes from the mapping cone formula for rational surgeries \cite{OSz2005-3}. Indeed,  a direct calculation shows that $\rk\HFhat(S^3_{1/q}(K'))\ge 2q-1$, with equality if and only if $K'$ is the trefoil. Moreover, when the genus of $K'$ is greater than 1, $\rk\HFhat(S^3_{1/q}(K'))\ge 4q+1$.  \end{remark}
%\piccaption[r]{\label{pic:unframedtrefoiltangle}}
\parpic[r]{$\begin{array}{c}\includegraphics[scale=0.28]{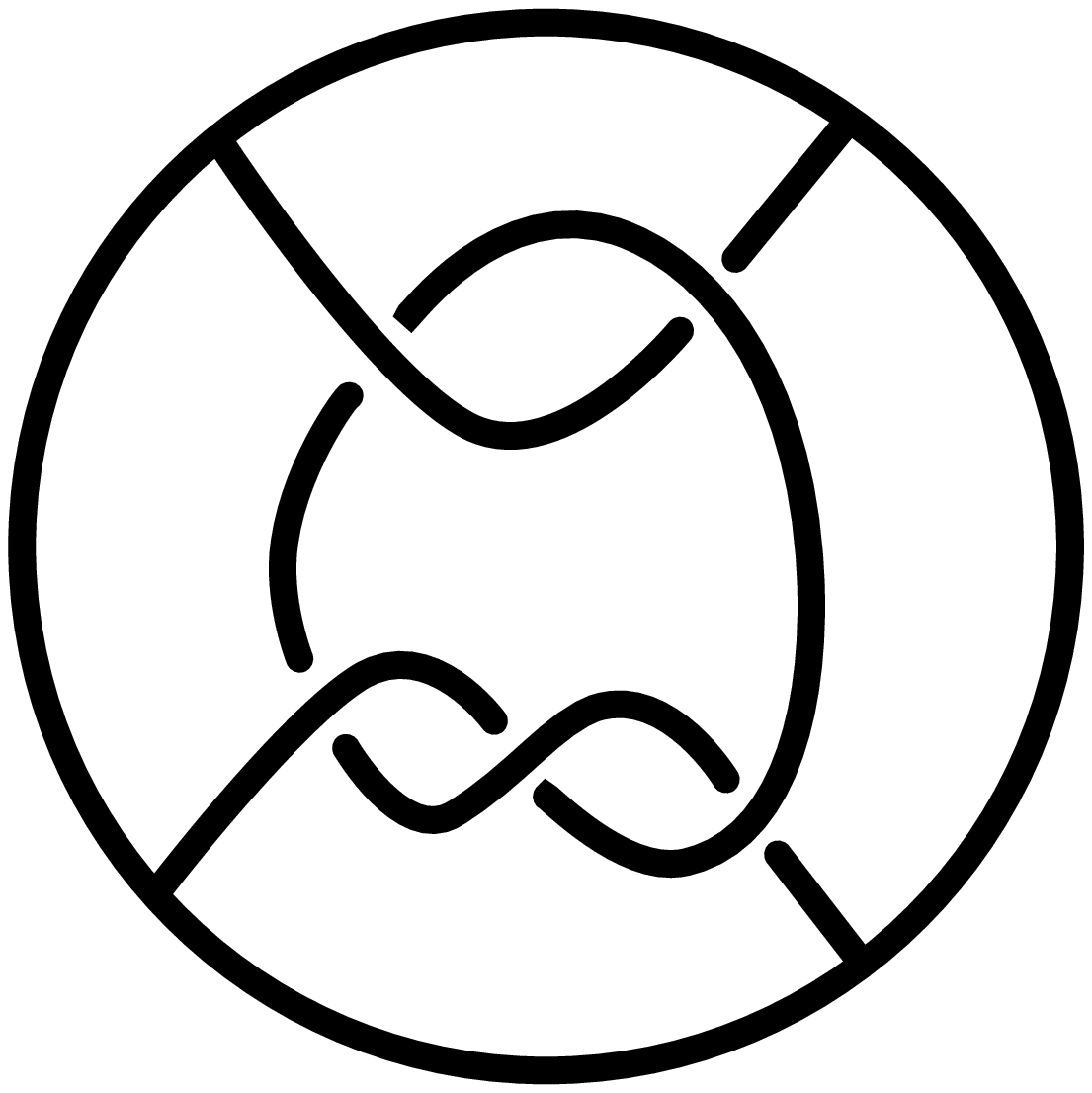}\end{array}$}
Thus we are left to deal with the case when $K'$ is the trefoil. This is a strongly invertible knot, and the tangle associated to the quotient of $S^3\smallsetminus\nu({K'})$ has the form shown on the right. This can be worked out by hand as in \cite{Bleiler1985}, for example. Alternatively, we remark that this tangle is a sum of two rational tangles: This reflects the Seifert fiber structure in the cover. Indeed, the two tangles lift to a pair of solid tori identified along an essential annulus, the cores of which are singular fibers of order 2 and 3. That this is the unique such tangle follows from the fact that torus knots admit a unique strong inversion \cite{Schreier1924}. From here it is straightforward to identify the $(-2,3,5)$-pretzel knot (the knot $10_{124}$ in Rolfsen's table \cite{Rolfsen1976}) as the appropriate $K$ for which  $\Sigma_K^2\cong S^3_1(K')$. The conclusion now follows by direct calculation: {\tt KhoHo} \cite{KhoHo} confirms that $\rk\Khred(10_{124})=7$, proving Theorem \ref{thm:main}. 

\section{Satellite knots whose Khovanov homology detect the unknot}
\label{sec:cor}

An interesting consequence of Theorem \ref{thm:main} is that it allows us to show that the Khovanov homology of many satellite constructions detects the unknot.     To make this precise, recall that associated to any non-trivial knot in a solid torus $P\hookrightarrow V\cong S^1\times D^2$, we obtain an operation
$$ P(-): \{\mathrm{Knots}\} \longrightarrow \{\mathrm{Knots}\},$$
\noindent where $P(K)$ is defined as the image of $P$ under an identification of $V$ with $\nu(K)$.\footnote{This identification requires a choice of framing which we obtain from a Seifert surface for $K$.}  $P(K)$ is called a {\em satellite} of $K$ with  {\em pattern} $P$.  

Observe that if $K_1\simeq K_2$ then  $P(K_1)\simeq P(K_2)$, so that the operation defined by $P$ does indeed descend to isotopy classes of knots.  Given an invariant of a knot, $K$, this observation allows us to define infinite families of invariants:  Simply apply the invariant to all the various satellites of $K$.  

In the case at hand, the invariant we are considering is the reduced Khovanov homology.  Suppose that we choose a pattern $P$ so that $P(K)$ has tangle unknotting number one for {\em every} knot $K$, and so that $P(K)\simeq U$, if and only if $U$ is the unknot.   In this situation, Theorem \ref{thm:main} applies to show that $\rk\Khred(P(K))=1$ if and only if $P(K)\simeq U$ which, in turn, happens if and only if $K\simeq U$.  Thus we arrive at Corollary \ref{cor:satellite}.

\labellist 
	\pinlabel $\cdot\!\cdot\!\cdot$ at 435 417 	
	\pinlabel {$\underbrace{\phantom{aaaaaaaa}}_n$} at 395 375
\endlabellist
\piccaptiontopside{\label{pic:satellitetorus}}
\parpic[r]{$\begin{array}{c}\includegraphics[scale=0.28]{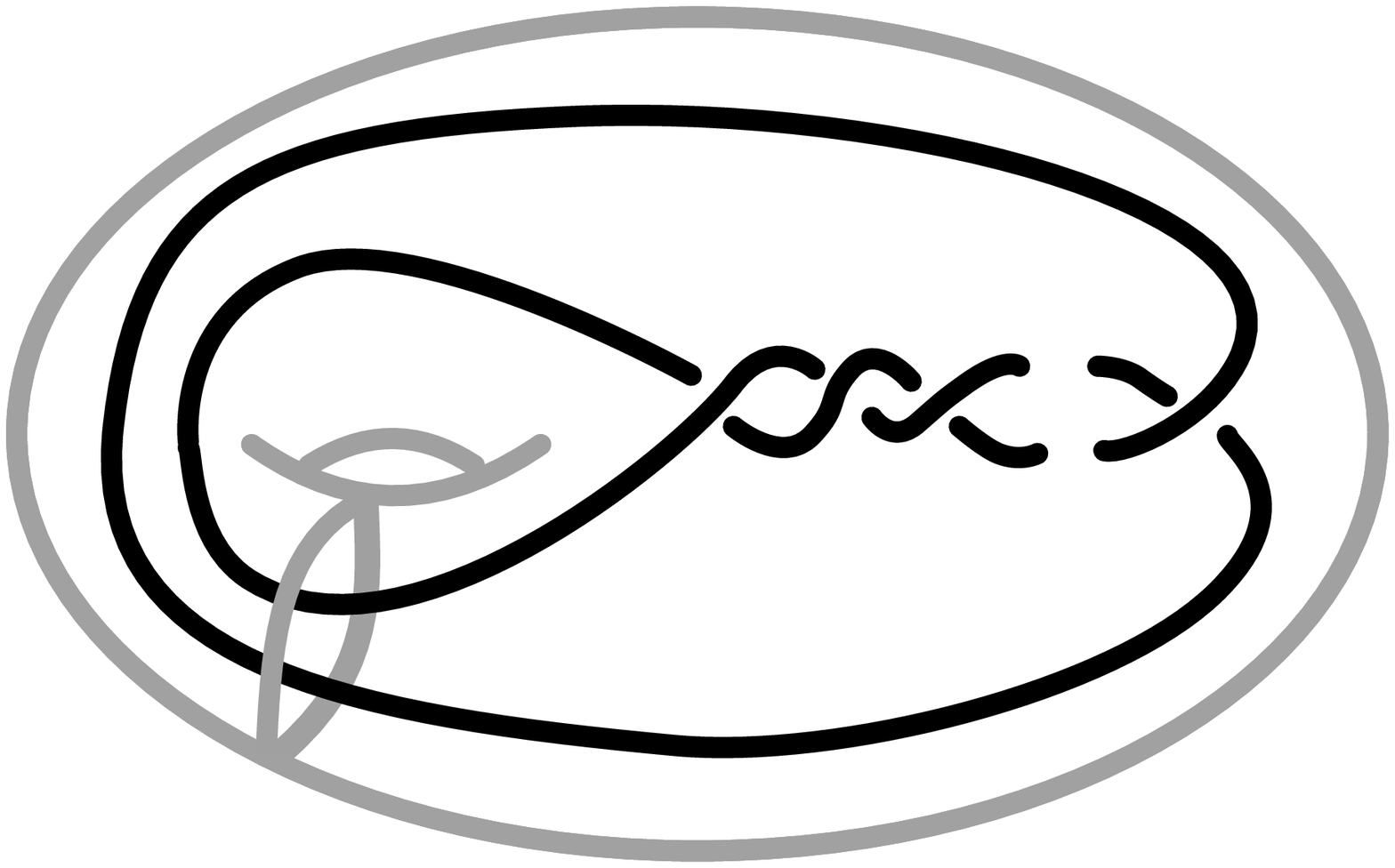}\end{array}$}

A simple infinite family of satellite constructions whose Khovanov homologies detect the unknot are provided by the patterns shown in the solid torus  on the right, %(see for example Rolfsen \cite{Rolfsen1976}) 
where $n$ denotes the number of half twists. It is straightforward to verify that each of these patterns satisfies the hypotheses of Corollary \ref{cor:satellite}.   Note that the $(2,\pm1)$-cable of $K$ is obtained for $n=\pm1$. Similarly, the positive (respectively negative) clasp, untwisted Whitehead double of $K$ is obtained for $n=2$ (respectively $n=-2$). The case $n=0$ is always the unknot, while the convention $n=\frac{1}{0}$ gives rise to the $2$-cable of the knot $K$ (this latter satellite is a link, and was handled by a different technique in \cite{Hedden2008}).

We remark that for  the satellites specified by the figure, it is straightforward to determine the knot in $S^3$ on which one performs surgery to obtain the double branched covers.  If we let $K_n$ denote the satellite of $K$ with pattern given by the figure with $n$ half-twists, then  \[\Sigma_{K_n}^2\cong S^3_{1/n}(K\# K^r)\] where $K^r$ denotes $K$ with the orientation reversed. Indeed, there is an obvious strong inversion on $K\# K^r$ exchanging the two summands. Extending this strong inversion over the surgery torus yields a strong inversion on $S^3_{1/n}(K\# K^r)$. From this, one can see that the quotient is $S^3$ and the image of the fixed-point set is $K_n$. See Akbulut and Kirby  \cite{AK1980} or Montesinos and Whitten \cite{MW1986} for details.   Coupled with Remark \ref{rmk:main}, this yields the explicit bound
$$\rk\Khred(K_n)\ge 4n+1,$$
whenever $K$ is non-trivial. This bound, together with the skein exact sequence for Khovanov homology \cite{Khovanov2000,Rasmussen2005}, recovers the fact that the reduced Khovanov homology of the $2$-cable of a knot detects the unknot \cite{Hedden2008}.

\bibliographystyle{plain}
\bibliography{mybib}

\begin{thebibliography}{10}

\bibitem{AK1980}
Selman Akbulut and Robion Kirby.
\newblock Branched covers of surfaces in 4-manifolds.
\newblock {\em Math. Ann.}, 252(2):111--131, 1980.

\bibitem{Bleiler1985}
Steven~A. Bleiler.
\newblock Prime tangles and composite knots.
\newblock In {\em Knot theory and manifolds (Vancouver, B.C., 1983)}, volume
  1144 of {\em Lecture Notes in Math.}, pages 1--13. Springer, Berlin, 1985.

\bibitem{Ghiggini2007}
Paolo Ghiggini.
\newblock {Knot Floer homology detects genus-one fibred knots}.
\newblock To appear in {\em Amer. J. Math.}

\bibitem{Hedden2008}
Matthew Hedden.
\newblock {Khovanov homology of the 2-cable detects the unknot}.
\newblock {arxiv.org/abs/0805.4418v1}.

\bibitem{Hedden2007}
Matthew Hedden.
\newblock {On Floer homology and the Berge conjecture on knots admitting lens
  space surgeries}.
\newblock {math.GT/0710.0357}.

\bibitem{Jones1985}
Vaughan F.~R. Jones.
\newblock A polynomial invariant of knots via {von Neumann} algebras.
\newblock {\em Bul. Amer. Math. Soc.}, 12:103--111, 1985.

\bibitem{Khovanov2000}
Mikhail Khovanov.
\newblock A categorification of the {J}ones polynomial.
\newblock {\em Duke Math. J.}, 101(3):359--426, 2000.

\bibitem{Khovanov2003}
Mikhail Khovanov.
\newblock Patterns in knot cohomology. {I}.
\newblock {\em Experiment. Math.}, 12(3):365--374, 2003.

\bibitem{KMOSz2007}
P.~B. Kronheimer, T.~S. Mrowka, P.~S. Ozsv{\'a}th, and Z.~Szab{\'o}.
\newblock Monopoles and lens space surgeries.
\newblock {\em Ann. of Math.}, 165(2):457--546, 2007.

\bibitem{Lickorish1981}
W.~B.~Raymond Lickorish.
\newblock Prime knots and tangles.
\newblock {\em Trans. Amer. Math. Soc.}, 267(1):321--332, 1981.

\bibitem{MW1986}
Jos{\'e}~M. Montesinos and Wilbur Whitten.
\newblock Constructions of two-fold branched covering spaces.
\newblock {\em Pacific J. Math.}, 125(2):415--446, 1986.

\bibitem{Ni2007}
Yi~Ni.
\newblock Knot {F}loer homology detects fibred knots.
\newblock {\em Invent. Math.}, 170(3):577--608, 2007.

\bibitem{OSz2004}
Peter Ozsv{\'a}th and Zolt{\'a}n Szab{\'o}.
\newblock Holomorphic disks and topological invariants for closed
  three-manifolds.
\newblock {\em Ann. of Math. (2)}, 159(3):1027--1158, 2004.

\bibitem{OSz2005-1}
Peter Ozsv{\'a}th and Zolt{\'a}n Szab{\'o}.
\newblock On the {H}eegaard {F}loer homology of branched double-covers.
\newblock {\em Adv. Math.}, 194(1):1--33, 2005.

\bibitem{OSz2005-3}
Peter~S. Ozsv{\'a}th and Zolt{\'a}n Szab{\'o}.
\newblock {Knot Floer homology and rational surgeries}.
\newblock {math.GT/0504404}.

\bibitem{OSz2006}
Peter~S. Ozsv{\'a}th and Zolt{\'a}n Szab{\'o}.
\newblock {The Dehn surgery characterization of the trefoil and the figure
  eight knot}.
\newblock {math.GT/0604079}.

\bibitem{Rasmussen2007}
Jacob Rasmussen.
\newblock {Lens space surgeries and L-space homology spheres}.
\newblock {math.GT/0710.2531}.

\bibitem{Rasmussen2005}
Jacob Rasmussen.
\newblock Knot polynomials and knot homologies.
\newblock In {\em Geometry and topology of manifolds}, volume~47 of {\em Fields
  Inst. Commun.}, pages 261--280. Amer. Math. Soc., Providence, RI, 2005.

\bibitem{Rolfsen1976}
Dale Rolfsen.
\newblock {\em Knots and links}.
\newblock Publish or Perish Inc., Berkeley, Calif., 1976.
\newblock Mathematics Lecture Series, No. 7.

\bibitem{Schreier1924}
Otto Schreier.
\newblock {\"Uber die Gruppen $A^aB^b=1$}.
\newblock {\em Abh. Math. Sem. Univ. Hamburg}, 3:167--169, 1924.

\bibitem{KhoHo}
Alexander Shumakovitch.
\newblock \url{KhoHo} -- a program for computing and studying Khovanov
  homology, \url{http://www.geometrie.ch/KhoHo}.

\bibitem{Waldhausen1969}
Friedhelm Waldhausen.
\newblock \"{U}ber {I}nvolutionen der {$3$}-{S}ph\"are.
\newblock {\em Topology}, 8:81--91, 1969.

\end{thebibliography}

\end{document}